\documentclass[10pt]{amsart}
\usepackage{amssymb}
\usepackage{epsfig}
\usepackage{url}
\usepackage{setspace}
\theoremstyle{plain}

\newtheorem{thm}{Theorem}[section]
\newtheorem{cor}[thm]{Corollary}
\newtheorem{lem}[thm]{Lemma}

\newtheorem{rem}[thm]{Remark}
\newtheorem{ques}[thm]{Question}

\newtheorem{exam}[thm]{Example}
\def\cal{\mathcal}
\def\bbb{\mathbb}
\def\op{\operatorname}
\renewcommand{\phi}{\varphi}
\newcommand{\R}{\bbb{R}}
\newcommand{\N}{\bbb{N}}
\newcommand{\Z}{\bbb{Z}}
\newcommand{\Q}{\bbb{Q}}

\begin{document}

\title[Diophantine systems involving symmetric polynomials]{A note on Diophantine systems involving three symmetric polynomials}
\author{Maciej Ulas}

\thanks{Research of the author was supported by Polish Government funds for science, grant IP 2011 057671 for the years 2012--2013.}

\begin{abstract}
Let $\overline{X}_{n}=(x_{1},\ldots,x_{n})$ and  $\sigma_{i}(\overline{X}_{n})=\sum x_{k_{1}}\ldots x_{k_{i}}$ be $i$-th elementary symmetric polynomial. In this note we prove that there are infinitely many triples of integers $a, b, c$ such that for each $1\leq i\leq n$ the system of Diophantine equations
\begin{equation*}
  \sigma_{i}(\overline{X}_{2n})=a, \quad \sigma_{2n-i}(\overline{X}_{2n})=b, \quad \sigma_{2n}(\overline{X}_{2n})=c
\end{equation*}
has infinitely many rational solutions. This result extend the recent results of Zhang and Cai, and the author. Moreover, we also consider some Diophantine systems involving sums of powers. In particular, we prove that for each $k$ there are at least $k$ $n$-tuples of integers with the same sum of $i$-th powers for $i=1,2,3$. Similar result is proved for $i=1,2,4$ and $i=-1,1,2$.
\end{abstract}

\maketitle

\section{Introduction}\label{sec1}

Let $\overline{X}_{n}=(x_{1},\ldots,x_{n})$ and let $\sigma_{i}(\overline{X}_{n})=\sum x_{k_{1}}\ldots x_{k_{i}}$ be the $i$-th elementary symmetric polynomial. In a recent paper \cite{Ul} we generalized the results of Zhang and Cai \cite{ZC,ZC2} and that of Schinzel \cite{Sch} by proving that for all $n\geq 4$ and each choice of $i,j\in\{1,\ldots,n\}$ with $i<j$ there are infinitely many rational numbers $a, b$ such that the system of the Diophantine equations
\begin{equation*}
\sigma_{i}(\overline{X}_{n})=a, \quad \sigma_{j}(\overline{X}_{n})=b,
\end{equation*}
has infinitely many solutions in integers. In fact, in the case of $j=n$ we proved that for all $i\in\{1,\ldots,n-1\}$ and each $a, b\in\Q\setminus\{0\}$ the above system has infinitely many solutions in rational numbers. With this paper  we have also started to study the systems of Diophantine equations involving three different elementary symmetric polynomials by proving that for each $n\geq 5$ there are infinitely many triples of rational numbers $a, b, c$ such that the system
\begin{equation*}
\sigma_{1}(\overline{X}_{n})=a, \quad \sigma_{2}(\overline{X}_{n})=b, \quad \sigma_{3}(\overline{X}_{n})=c
\end{equation*}
has infinitely many solutions in rational numbers $x_{i}, i=1,2,\ldots,n$. As a corollary we get that for each $k$, there are at least $k$ $n$-tuples of integers with the same sum, the same value of the second elementary symmetric polynomials and the same value of the third elementary symmetric polynomial. This result motivated us to raise the following question:

\begin{ques}[Question 5.6 in \cite{Ul}]\label{ques}
Let $n\geq 5$ and $1\leq i_{1}<i_{2}<i_{3}\leq n$ be given. Is it possible to find rational numbers $a, b, c$ such that the system
\begin{equation*}
\sigma_{i_{1}}(\overline{X}_{n})=a,\quad \sigma_{i_{2}}(\overline{X}_{n})=b,\quad \sigma_{i_{3}}(\overline{X}_{n})=c
\end{equation*}
has infinitely many rational solutions? (In the case of $i_{3}=n$ we assume that $c\neq 0$.)
\end{ques}

In this paper we continue this line of research and show that the Question \ref{ques} has positive answer for the triples of the form $(i_{1},i_{2},i_{3})=(i,2n-i,2n)$ with given $i$ satisfying the condition $1\leq i<n$. This is the first case where the problem is considered for three  equations where the indices of the elementary symmetric polynomials are not fixed. This result is obtained with the help of an interesting identity involving symmetric polynomials (Lemma \ref{simplelem}). This allows us to reduce the system of three equations to one more convenient equation that can be more easily handled. The main result of this paper (together with Lemma \ref{simplelem}) is proved in Section \ref{sec2}. In the last section we prove that for each $k\in\N$ there are at least $k$ $n$-tuples of integers with the same sum of $i$-th powers for $i=1,2,3$. Similar result is proved for exponents $i=1,2,4$ and $i=-1,1,2$. These results motivated us to state a general question concerning the existence of rational solutions of systems of Diophantine equations involving sums of powers.

\begin{rem}
{\rm
In the paper we use the standard convention that $\sigma_{i}(\overline{X}_{n})$ is equal to 1 for $i=0$ and $0$ for $i<0$. All computations in this paper were performed with the assistance of {\sc Mathematica 7} \cite{Wol}.}
\end{rem}

\section{A simple lemma and the main result}\label{sec2}

We start with a simple identity involving symmetric functions which will be useful in the proof of our main result.

\begin{lem}\label{simplelem}
Let $i, n$ be non-negative integers and suppose that $i\leq n$. Then we have
\begin{equation*}
\sigma_{i}\Big(\overline{X}_{n},\frac{1}{\overline{X}_{n}}\Big)= \sigma_{2n-i}\Big(\overline{X}_{n},\frac{1}{\overline{X}_{n}}\Big),
\end{equation*}
where $\frac{1}{\overline{X}_{n}}=\Big(\frac{1}{x_{1}},\ldots,\frac{1}{x_{n}}\Big)$.
\end{lem}
\begin{proof}
We consider an arbitrary summand $s_{I,J}$ of the rational function $\sigma_{i}\Big(\overline{X}_{n},\frac{1}{\overline{X}_{n}}\Big)$, and write
\begin{equation*}
s_{I,J}=\prod_{j\in I}x_{j}\prod_{j\in J}\frac{1}{x_{j}},
\end{equation*}
where $(I, J)\in\cal{A}_{i}$ and the set $\cal{A}_{i}$ contains all pairs $(I,J)$ of sets satisfying the conditions: $I, J\subset \{1,2,\ldots,n\}$ with $|I|+|J|=i$. For $(I,J)\in\cal{A}_{i}$ we define $I'=\{1,\ldots,n\}\setminus I$ and similarly the set $J'$. Note that $|I'|+|J'|=2n-i$. We also define the set $\cal{A}_{i}'$ which contains all corresponding pairs $(I',J')$, for $(I,J)\in\cal{A}_{i}$.  From the obvious equality
\begin{equation*}
\prod_{j=1}^{n}x_{j}\prod_{j=1}^{n}\frac{1}{x_{j}}=1,
\end{equation*}
we immediately deduce the following identity
\begin{equation*}
s_{I,J}=\prod_{j\in I}x_{j}\prod_{j\in J}\frac{1}{x_{j}}=\prod_{j\in I'}x_{j}\prod_{j\in J'}\frac{1}{x_{j}}=:s_{I',J'}.
\end{equation*}
However, the expression $s_{I',J'}$ is a summand in $\sigma_{2n-i}\Big(\overline{X}_{n},\frac{1}{\overline{X}_{n}}\Big)$, and the above identity shows that the correspondence $s_{I,J}\rightarrow s_{I',J'}$ is a bijection. We thus get the equality  \begin{equation*}
\sigma_{i}\Big(\overline{X}_{n},\frac{1}{\overline{X}_{n}}\Big)=\sum_{(I,J)\in \cal{A}_{i}}s_{I,J}=\sum_{(I',J')\in\cal{A}_{i}'}s_{I',J'}=
\sigma_{2n-i}\Big(\overline{X}_{n},\frac{1}{\overline{X}_{n}}\Big)
\end{equation*}
and our result follows.
\end{proof}

Now we are ready to prove the main result of this paper.

\begin{thm}\label{mainthm}
 Let $i, n$ be a positive integers with $i\leq n$ and $n\geq 2$. Let $t_{1},\ldots, t_{n-2}$ be rational parameters. Then the set of those $a\in\Q(t_{1},\ldots,t_{n-2})\setminus\{0\}$ such that the system of Diophantine equations
\begin{equation}\label{mainsys}
  \sigma_{i}(\overline{X}_{2n})=a, \quad \sigma_{2n-i}(\overline{X}_{2n})=a, \quad \sigma_{2n}(\overline{X}_{2n})=1
\end{equation}
has infinitely many solutions in $\Q(t_{1},\ldots,t_{n-2})$, is infinite.
\end{thm}
\begin{proof}
The proof of our result is rather technical at certain points and some symbolic computations are omitted. The reason is simple: we want to get the result which is as general as possible. More precisely, we are interesting in constructing solutions which depend on $n-2$ rational parameters. In particular, this allow us to get an interesting result given in the Corollary \ref{cor} which has strong Diophantine flavor. Unfortunately, the degree of generality implies that the numerical data practically cannot be given explicitly. In such situations we only present a principal argument, omitting all the tiresome details. However, later on we work out an example explicitly, where the reader can check all details (see Example \ref{exam1}).

Let us fix $i\leq n$. In order to find $a$'s which satisfy the required property we put $x_{n+j}=\frac{1}{x_{j}}$ for $j=1,\ldots,n$ and thus by Lemma \ref{simplelem} we have the equality $\sigma_{i}(\overline{X}_{2n})=\sigma_{2n-i}(\overline{X}_{2n})$. Moreover, the third equation from the system (\ref{mainsys}) is satisfied too. Therefore, we need to solve only the one equation
\begin{equation*}
  \sigma_{i}(\overline{X}_{2n})=\sigma_{i}\Big(\overline{X}_{n},\frac{1}{\overline{X}_{n}}\Big)=a
\end{equation*}
which can be rewritten as
\begin{equation*}
 \sigma_{i}\Big(\overline{X}_{n-2},\frac{1}{\overline{X}_{n-2}},P,Q,\frac{1}{P},\frac{1}{Q}\Big)=a,
\end{equation*}
where in order to shorten the notation we put $x_{n-1}=P, x_{n}=Q$.
Now let us take
\begin{equation*}
a=\sigma_{i}\Big(\overline{T}_{n-2},\frac{1}{\overline{T}_{n-2}},p,q,\frac{1}{p},\frac{1}{q}\Big),
\end{equation*}
where $\overline{T}_{n-2}=(t_{1},\ldots,t_{n-2})$ is a vector of rational parameters and $p, q$ are rational parameters too. Now, let us put $\overline{X}_{n-2}=\overline{T}_{n-2}$ and note that the equation
\begin{equation*}
\sigma_{i}\Big(\overline{T}_{n-2},\frac{1}{\overline{T}_{n-2}},P,Q,\frac{1}{P},\frac{1}{Q}\Big)=\sigma_{i}\Big(\overline{T}_{n-2},\frac{1}{\overline{T}_{n-2}},p,q,\frac{1}{p},\frac{1}{q}\Big)
\end{equation*}
defines a curve over the rational function field $K:=\Q(p,q,\overline{T}_{n-2})$ with two known $K$-rational points $(P,Q)=(p^{\epsilon},q^{\epsilon})$, where $\epsilon\in\{-1,1\}$.


We immediately get that our equation is equivalent to
\begin{equation}\label{eq2}
\sum_{k=0}^{4}u_{k}\sigma_{k}\Big(P,Q,\frac{1}{P},\frac{1}{Q}\Big)=\sum_{k=0}^{4}u_{k}\sigma_{k}\Big(p,q,\frac{1}{p},\frac{1}{q}\Big),
\end{equation}
where in order to shorten the notation we put (remember we fixed $i$)
\begin{equation*}
u_{k}=\sigma_{i-k}\Big(\overline{T}_{n-2},\frac{1}{\overline{T}_{n-2}}\Big).
\end{equation*}
Indeed, the above equivalence follows from the general identity
\begin{equation*}
\sigma_{i}(\overline{X}_{n+j})=\sum_{k=0}^{i}\sigma_{i-k}(\overline{X}_{n})\sigma_{k}(x_{n+1},\ldots,x_{n+j}),
\end{equation*}
where $i, j, n$ are non-negative integers with $i\leq n+j$. The right hand side is just a simple expansion of $\sigma_{i}(\overline{X}_{n+j})$ with respect to the variables $x_{n+1},\ldots,x_{n+j}$.

We return to the equation (\ref{eq2}). If $k=0$ then the initial terms of sums on both sides of our equation are equal so $u_{0}$ will be absent in our equation. Moreover, let us note that if $i-k<0$ then $u_{k}=0$, and if $i-k=0$ then $u_{k}=1$. In particular, if $n>i\geq 4$ then $u_{k}\neq 0$ for $k=0,\ldots,4$.
The equation (\ref{eq2}), after clearing the common denominator $pqPQ$, can be written as
$F_{P}(Q):=a_{0}Q^2+a_{1}Q+a_{2}=0$, where
\begin{align*}
 a_{0}&=p q (u_2P^2+(u_1+u_3)P+u_2),\\
 a_{1}&=pq(u_{1} + u_{3})(P^2+1)-P((p^2q^2+p^2+q^2+1)u_2+(u_1+u_3)(p+q)(pq+1)),\\
 a_{2}&=a_{0}.
\end{align*}

We thus see that in order to prove our result it is enough to show that the set of those $P\in K$ for which the equation $F_{P}(Q)=0$ has a $K$-rational solution (with respect to $Q$) is infinite. Equivalently, we need to show that there are infinitely many $P\in K$ such that the discriminant $\Delta(P)$ of the polynomial $F_{P}$ is a square in $K$. We thus consider the hyperelliptic quartic curve in the $(P,S)$ plane given by the equation
\begin{equation*}
\cal{C}:\;S^2=a_{1}(P)^2-4a_{0}(P)^2=:H(P).
\end{equation*}
Here we put $S=2a_{0}(P)Q+a_{1}(P)$. We see that $\cal{C}$ is defined over the field $K$.
In general we have $\op{deg}_{P}H=4$. However, we note that if $(u_1-2 u_2+u_3)(u_1+2 u_2+u_3)=0$ then $\op{deg}_{P}H\leq 3$. One can check that the discriminant $\op{Disc}_{P}(H)$ of the right-hand side of the polynomial defining the curve $\cal{C}$ is a non-zero element of the field $K$.
This implies that $\cal{C}$ is of genus $\leq  1$. Moreover, our construction of the expression for $a$ was performed in such a way that the curve $\cal{C}$ contains two non-trivial $K$-rational points
\begin{equation*}
U=(p,2a_{0}(p)q+a_{1}(p)),\quad\quad  V=\Big(\frac{1}{p},2a_{0}\Big(\frac{1}{p}\Big)\frac{1}{q}+a_{1}\Big(\frac{1}{p}\Big)\Big).
\end{equation*}

Let us note that it is possible to specialize parameters $p, q, t_{i}$ for $i=1,2,\ldots,n-2$, i.e. simple take $p, q, t_{i}$ as concrete rational numbers, in such a way that the genus of the specialized curve, say $\cal{C}'$, is 0. This is equivalent with finding rational solutions of the equation $\op{Disc}_{P}(H)=0$. After such specialization, the curve $\cal{C}'$ will be defined over the field $\Q$ and can be parameterized with the standard method of projection from $\Q$-rational point $U'$
(which is obtained by the specialization of the point $U$). However, in order to have $\op{Disc}_{P}(H)=0$ it is necessary to solve the equation which seems to be more difficult then the equation we are working on. Moreover, there is one more weakness of this approach: in this case we were able to construct the solutions of (\ref{mainsys}) which depend on one rational parameter only.

In the light of the remark above we need to concentrate on the case when $\cal{C}$ is of genus one. Then we can use the point $U$ as a point at infinity in order to find a birational map $\phi:\; \cal{C}\rightarrow \cal{E}$. Here $\cal{E}$ is given by the Weierstrass equation of the form $\cal{E}:\;Y^2=X^3+AX+B$ (see \cite[p. 77]{Mor}). We note that $A$ and $B$ are very complicated polynomials depending on variables $t_{1},\ldots, t_{n-2}, p$ and  $q$. Moreover, there is no $f\in\Z[\overline{T}_{n-2},p,q]$ such that $f^4|A$ and $f^6|B$. In order to finish the proof it is enough to prove that the point $W=(X,Y)=\phi(V)$ is of infinite order in the group of $K$-rational points lying on the curve $\cal{E}$. This was achieved with the help of a generalized version of the classical Nagell-Lutz theorem (see \cite[p. 78]{Sko} and \cite[p. 268]{Con}). The generalization states that if $E:\;Y^2=X^3+\cal{A}X+\cal{B}$ is an elliptic curve over the field $\Q(t_{1},\ldots,t_{k})$ of rational functions in $k$ variables with $\cal{A}, \cal{B}\in\Z[t_{1},\ldots,t_{k}]$ then the coordinates of torsion points on $E$ have coordinates in $\Z[t_{1},\ldots,t_{k}]$. This generalization can be used to prove that $W$ is of infinite order on $\cal{E}$. Indeed, the brute force calculation with the help of {\sc Mathematica} program \cite{Wol} reveals that both coordinates of $W$ are not polynomials.
This guarantees that $W$ is of infinite order on $\cal{E}(K)$. In particular, for all but finitely many points $W'\in \{[n]W:\;n\in\N\}$ the preimage $\phi^{-1}(W')$ corresponds to a point on the curve $\cal{C}$ and thus to a solution of the equation (\ref{eq2}), and finally to a solution of the system (\ref{mainsys}).

As was mentioned on the beginning of the proof, we do not present the details of these computations because they are long and not enlightening. However, in the Example \ref{exam1} below we present all the necessary computations where the interested reader can check all the details.
\end{proof}

\begin{rem}
{\rm Although we did not present all the messy details of the proof of the theorem above let us describe in a few words how one can construct a new solution of the equation  (\ref{eq2}) (and thus the solution of the system (\ref{mainsys})) using the point $U$ which lies on the curve $\cal{C}$. It is easy to find a parabola $S=\alpha P^2+\beta P+\gamma $ (here $\gamma$ is the $S$-th coordinate of the point $U$) which is tangent to the curve $\cal{C}$ at point $W$ with multiplicity three (or two in the case when $\op{deg}_{P}H=3$ and then we can take $\alpha=0$). In this circumstances the polynomial equation
\begin{equation*}
(\alpha P^2+\beta P+\gamma )^2=a_{1}(P)^2-4a_{0}(P)^2
\end{equation*}
is of degree four (or three in case of $\op{deg}_{P}H=3$) with a triple root at $P=p$ (respectively with a double root at $P=p$). This implies that the fourth (respectively third) root, say $P'$, must be $K$-rational and we get $[2]W=(P',\alpha P'^2+\beta P'+\gamma )$. However, even the expression for $P'$ is very complicated and we do not see any reason to present it explicitly. Let us also note that this method of construction of a new point, lying on hyperelliptic quartic, from an old one  essentially goes back to Euler (see also \cite[p. 69]{Mor}).}
\end{rem}

From the Theorem \ref{mainthm} we immediately deduce the following corollary.

\begin{cor}\label{cor}
Let $i, n$ be positive integers with $i\leq n$ and $n\geq 2$ and let us put $\overline{T}_{n-2}=(t_{1},\ldots,t_{n-2})$. Then for every positive integer $k$, there exist infinitely many primitive sets of $k$ $n$-tuples of polynomials from $\Z[\overline{T}_{n-2}]$ with the same value of $i$-th, $2n-i$-th elementary symmetric polynomial and the same product.
\end{cor}

Let us recall that a set $S$ of $n$-tuples is called primitive if the greatest common divisor of all elements of all $n$-tuples of $S$ is 1. We omit the proof of this corollary. The reason is simple: the proof goes exactly as in \cite[Corrollary 3.2]{Ul}, where similar result is proved for $n$-tuples of polynomials with the same value of $i$-th symmetric polynomial and the same product.

\begin{rem}
{\rm Let us note that the result contained in the corollary above is much stronger then the statement which says that for each $k\in\N$ there is at least $k$ $n$-tuples of integers with the same value of $i$-th, $2n-i$-th elementary symmetric polynomial and the same product.}
\end{rem}

\begin{exam}\label{exam1}
{\rm In order to present an example of solutions of the system (\ref{mainsys}) explicite let us take $i=1$ and $n=3$. Then $u_{1}=1, u_{2}=u_{3}=u_{4}=0$. Moreover, in order to shorten the presentation let us put $T_{1}=t_{1}=1, p=2$. Following the proof of the Theorem \ref{mainthm} we have $a=\sigma_{1}\Big(1,1,2,q,\frac{1}{2},\frac{1}{q}\Big)=\frac{2 q^2+9 q+2}{2 q}$, where $q$ is a rational parameter. We thus consider the equation
\begin{equation}\label{eq3}
\sigma_{1}\Big(1,1,P,Q,\frac{1}{P},\frac{1}{Q}\Big)=\frac{2 q^2+9 q+2}{2 q}.
\end{equation}
In this case the curve $\cal{C}$ takes the form
\begin{equation*}
\cal{C}:\;S^2=4q^2(P^4+1)-4q(q+2)(2q+1)P(P^2+1)+(4 q^4+20 q^3+25 q^2+20 q+4)P^2.
\end{equation*}
Using the point $(0,2q)$ which lie on $\cal{C}$ as a point at infinity we found the birational model of $\cal{C}$ given by Weierstrass equation $
\cal{E}:\;Y^2=X^3+27AX-54B$, where
\begin{align*}
A&=16 q^8+160 q^7+408 q^6+200 q^5-115 q^4+200 q^3+408 q^2+160 q+16,\\
B&=(4 q^4+20 q^3+q^2+20 q+4)(A-204q^4),
\end{align*}
with discriminant
\begin{equation*}
\Delta(\cal{E})=2^{16}3^{12}q^8 (q+2)^2 (2 q+1)^2(2 q^2-3 q+2)(2 q^2+13 q+2).
\end{equation*}
The curve $\cal{E}$ is defined over the rational function field $\Q(q)$ and contains a $\Q(q)$-rational torsion point of order two:
\begin{equation*}
T=(3(4 q^4+20 q^3+q^2+20 q+4),0).
\end{equation*}
The point $U=(2,4(q^2-1))$ maps through the $\phi$, where $\phi:\;\cal{C}\rightarrow\cal{E}$ is a birational map, to the point
\begin{equation*}
W=\phi(U)=(3(4 q^4+20 q^3+q^2-4 q+4),-216 (q-2) q (2 q+1)).
\end{equation*}
Because the square of the $Y$-th coordinate of the point $W$ do not divide $\Delta(\cal{E})$ we get that the point $W$ is of infinite order in the group $\cal{E}(\Q(q))$. This implies that the set of $\Q(q)$-rational points on $\cal{E}$ is infinite. Without difficulty we can find solutions of the equation (\ref{eq3}) now. For $k=2,3,\ldots $ we compute the point $[k]W=\sum_{i=1}^{k}W$ on the curve $\cal{E}$; next we calculate the corresponding point $\phi^{-1}([k]W)=(P_{k},S_{k})$ on $\cal{C}$. Here $S_{k}=2a_{0}(P_{k})Q+a_{1}(P_{k})$ and thus $Q_{k}:=(S_{k}-a_{1}(P_{k}))/2a_{0}(P_{k})$ solves the equation $F_{P_{k}}(Q)=0$. For example the point $[2]W$ corresponds to the solution of the equation (\ref{eq3}) given by
\begin{equation*}
P=\frac{3 (q-2) q}{2 (q-1) (q+1) (2 q-1)},\quad Q=-\frac{3 q (2 q-1)}{2 (q-2) (q-1) (q+1)}.
\end{equation*}
The point $[3]W$ gives the solution
\begin{align*}
P&=\frac{(4 q^3-2 q^2-7 q+8)(8 q^3-7 q^2-2 q+4)}{2(q^3+4 q^2-4 q+2)(2 q^3-4 q^2+4 q+1)},\\
Q&=\frac{(q^3+4 q^2-4 q+2)(8 q^3-7 q^2-2 q+4)}{q(2q^3-4 q^2+4 q+1)(4 q^3-2 q^2-7 q+8)},
\end{align*}
and so on.

A question arises how big is the set of those $q\in\Q$ such that the curve $\cal{E}_{q_{0}}$ obtained from $\cal{E}$ by the specialization at rational number $q=q_{0}$ has positive rank? The answer follows from Silverman's theorem which states that if $\cal{E}$ is an elliptic curve
defined over $\Q(q)$ with positive rank, then for all but finitely many $q_{0}\in\Q$, the curve $\cal{E}_{q_{0}}$ obtained from the
curve $\cal{E}$ by specialization at $q=q_{0}$ has positive rank \cite[p. 368]{Sil}. From this result we see that for all but finitely many $q_{0}\in\Q$, the
elliptic curve $\cal{E}_{q_{0}}$ is of positive rank. In fact, we checked that if $q_{0}\in\Q\setminus\{-1,1,2,\frac{1}{2}\}$ then the point $W_{q_{0}}$ which is specialization of the point $W$ at $q=q_{0}$ is of infinite order on the curve $\cal{E}_{q_{0}}$. This result was obtained by looking for rational roots of the denominators arising  in the calculation of $[k]W$ for $k=1,2\ldots,12$ (from Mazur's theorem we know that if $W_{q_{0}}$ is of finite order then $k\leq 12$). The interesting thing is that this remark and the shape of $a=a(q)=\frac{2 q^2+9 q+2}{2 q}$ immediately implies that the set of those $a\in\Q$ for which there are infinitely many rational solutions of the equation
\begin{equation*}
\sigma_{1}\Big(1,1,P,Q,\frac{1}{P},\frac{1}{Q}\Big)=\sigma_{5}\Big(1,1,P,Q,\frac{1}{P},\frac{1}{Q}\Big)=a,
\end{equation*}
is dense subset of $\R$. Indeed, we have $\lim_{q\rightarrow \pm \infty}a(q)=\pm \infty$ and the continuity of $a(q)$ implies that $\overline{a(\Q)}=\R$, where the closure is taken in the Euclidean topology.

}
\end{exam}

\section{Diophantine systems involving sums of powers}\label{sec4}

A question arises whether it is possible to get some new results for systems of Diophantine equations involving different types of symmetric polynomials. It seems that from Diophantine point of view the most interesting is a family involving sums of powers. Let $n$ be a positive integer and for $i\in\Z\setminus\{0\}$ let us put
\begin{equation*}
s_{i}(\overline{X}_{n})=\sum_{j=1}^{n}x_{j}^{i}.
\end{equation*}
Note that the exponent in our notation of $s_{i}$ is not necessarily a positive integer. We are thus interested in the Diophantine systems of the form
\begin{equation}\label{generalsystem}
s_{e_{i}}(\overline{X}_{n})=a_{i}\quad\mbox{for}\quad i=1,2,\ldots,m,
\end{equation}
where $e_{1},\ldots,e_{m}$ are given non zero integers and $a_{1},\ldots, a_{m}$ are rational numbers. It is well known that the above system with $m=2, e_{1}=1, e_{2}=2$ and the suitable choice of  $a_{1}, a_{2}\in\Q$ has infinitely many solutions in rational numbers for each $n\geq 3$. Indeed, we can eliminate $x_{n}$ from the first equation and choose $a_{1}, a_{2}$ such that the resulting hypersurface of degree two has a rational point $P$. Using then a projection form the point $P$ we can find parametric solution of our system. Similar argument works in case $e_{1}=1, e_{2}=3$. Then the resulting variety is just a cubic hypersurface.  The known rational point and the chord and tangent procedure allows to produce infinitely many rational points (in case of $n=3$) or parametric solution (in case $n>3$). A bit more of work allows us to cover the case of $e_{1}=1, e_{2}=4$ too (with $n=4$). It would be very interesting to prove that for each pair of the form $(e_{1},e_{2})=(1,k)$, where $k$ is given non-zero integer there is a $n\in\N$ and rational numbers $a_{1}, a_{2}$ such that the system (\ref{generalsys}) (with $m=2$) has infinitely many rational solutions.

In the case of system (\ref{generalsystem}) with three equations and certain triples of exponents we present the following result.

\begin{thm}\label{123thm}
Let $(e_{1},e_{2},e_{3})\in\{(1,2,3),(1,2,4)\}$ and let $n\geq 4$. For each $a\in\Q\setminus\{0\}$ there are non-zero rational numbers $b, c$ such that the system of Diophantine equations
\begin{equation}\label{123sys}
s_{i_{1}}(\overline{X}_{n})=a,\quad s_{i_{2}}(\overline{X}_{n})=b,\quad s_{i_{3}}(\overline{X}_{n})=c
\end{equation}
has infinitely many rational solutions $x_{1},\ldots,x_{n}$. Moreover, if $(e_{1},e_{2},e_{3})=(-1,1,2)$ then for each $a, b\in\Q\setminus\{0\}$ there exists a non-zero rational number $c$ such that the system {\rm (\ref{123sys})} has infinitely many rational solutions.
\end{thm}
\begin{proof}
We first prove our result in the case $n=4$ and then deduce the solution of (\ref{123sys}) for all $n\geq 5$.
Let us fix $a\neq 0$. First of all let us note that the variety defined by the system (\ref{123sys}) defines an affine curve in $\R^{4}$. Let us call this curve by $V$. The degree of this curve is 6 in case of $(e_{1},e_{2},e_{3})=(1,2,3)$ and it is equal to 8 in case of $(e_{1},e_{2},e_{3})=(1,2,4), (-1,1,2)$. We are interesting in finding those rational numbers $b, c$ such that $V$ is of genus $\leq 1$. Because of the Faltings theorem (which states that curves with genus $>1$ has only finitely many rational points \cite{Fal}) this is the only case when we can expect that the curve $V$ has infinitely many rational points.

We first consider the case $(e_{1},e_{2},e_{3})=(1,2,3)$. We eliminate $x_{4}$ from the first equation in (\ref{123sys}), i.e. $x_{4}=a-x_{1}-x_{2}-x_{3}$ and put this expression into the second and third equation. We thus consider the curve $V'$ defined by the intersection of a quadric and a cubic surface
\begin{equation*}
V':\;\begin{cases}\begin{array}{ccc}
       x_{1}^2+x_{2}^2+x_{3}^2+(a-x_{1}-x_{2}-x_{3})^2 & = & b, \\
       x_{1}^3+x_{2}^3+x_{3}^3+(a-x_{1}-x_{2}-x_{3})^3 & = & c.
     \end{array}
     \end{cases}
\end{equation*}
Let $F_{i}(x_{1},x_{2},x_{3})$ be $i$-th polynomial defining the curve $V'$ for $i=1,2$. The necessary condition for $V'$ to have rational points is vanishing of the resultant of $F_{1}, F_{2}$ with respect to, say, $x_{3}$. We have $\op{Res}_{x_{3}}(F_{1},F_{2})=F(x_{1},x_{2})^2$, where
\begin{equation*}
F=a^3+\sigma_1(3b-3 a^2)-3ab+6a\sigma_1^2-6a\sigma_2+2c-6\sigma_1^3+12\sigma _1\sigma_2,
\end{equation*}
and $\sigma_{1}(x_{1},x_{2})=x_{1}+x_{2}, \sigma_{2}(x_{1},x_{2})=x_{1}x_{2}$. For most choices of $b, c\in\Q$ the genus of the curve defined by the equation $C:\;F(x_{1},x_{2})=0$ is equal to 1. However, in order to guarantee the existence of infinitely many rational points on $V'$ we need to have genus of $C$ equal to 0. We thus need to impose conditions on $b, c\in\Q$ to ensure that $C$ has rational rational singular points. In order to do that we compute the intersection of the ideal $I=<F(x_{1},x_{2}),\partial_{x_{1}}F(x_{1},x_{2}),\partial_{x_{2}}F(x_{1},x_{2})>$ with $\Z[a,b,c]$. We get the equality $I\cap \Z[a,b,c]=<f_{1}(a,b,c)f_{2}(a,b,c)>$, where
\begin{equation*}
f_{1}=a^3 - 6 a b + 8 c,\quad f_{2}=a^6 - 12 a^4 b + 39 a^2 b^2 - 16 b^3 +12(a^2-6b)ac + 48 c^2.
\end{equation*}
We thus see that $C$ is singular if $f_{1}=0$ or $f_{2}=0$. We consider the case $f_{1}=0$ first and get that $c=\frac{1}{8}a(6b-a^{2})$. With the $c$ defined in this way the curve $C$ is reducible which follows from the factorization of $F$
\begin{equation*}
F(x_{1},x_{2})=\frac{3}{4}(a-2(x_1+x_2))(a^2-2b-2a(x_1+x_2)+4(x_1^2+x_2^2)).
\end{equation*}
We thus take $x_{2}=\frac{a}{2}-x_{1}$ and get
\begin{align*}
F_{1}\Big(x_{1},\frac{a}{2}-x_{1},x_{3}\Big)&=\frac{1}{2}(a^2-2b-2a(x_1+x_3)+4x_1^2+4x_3^2),\\
F_{2}\Big(x_{1},\frac{a}{2}-x_{1},x_{3}\Big)&=\frac{3}{8}a(a^2-2b-2a(x_1+x_3)+4x_1^2+4x_3^2).
\end{align*}
We thus see that our theorem will be proved if the conic $Q:\;a^2-2b-2a(x_1+x_3)+4x_1^2+4x_3^2=0$ has infinitely many rational points. Let us put
$b=\frac{1}{8}(3a^2+d^2)$, where $d$ is a rational number. With this choice of $b$ the conic $Q$ has a rational point $(x_{1},x_{2})=\Big(0,\frac{a+d}{4}\Big)$ and thus $Q$ can be parameterized by rational functions. The parametrization is given by
\begin{equation*}
x_{1}=\frac{2 (4 a-d t)}{t^2+16},\quad x_{3}=\frac{(t+4) ((a-d)t+4(a+d))}{4 \left(t^2+16\right)}.
\end{equation*}
Using these expressions we get expressions for $x_{2}$ and $x_{4}$ in the form
\begin{equation*}
x_{2}=\frac{t (a t+4 d)}{2(t^2+16)},\quad x_{4}=\frac{(t-4) ((a+d)t+4(d-a))}{4 \left(t^2+16\right)}.
\end{equation*}
Summing up, with $x_{i}$ for $i=1,2,3,4$ given above we get
\begin{equation}\label{specsol1}
s_{1}(\overline{X}_{4})=a,\quad s_{2}(\overline{X}_{4})=\frac{1}{8}(3a^2+d^2),\quad s_{3}(\overline{X}_{4})=\frac{1}{32}a(5 a^2 + 3 d^2)
\end{equation}
for each $t\in\Q$. We should also note that we can choose an infinite set $\cal{A}\subset \Q\times\Q$ such that for each $(a, d)\in\cal{A}$ the system (\ref{123sys}) has infinitely many positive rational solutions. Indeed, a quick investigation of the obtained expressions for $x_{i}$
shows that if $0<d<a$ and
\begin{equation*}
0<t<\frac{4(a-d)}{a+d}\quad\mbox{or}\quad 4<t<\frac{4a}{d}
\end{equation*}
then $x_{i}>0$ for $i=1,2,3,4$. This proves the first part of our theorem.

Let us also note that we could try to solve the equation $f_{2}=0$ in order to find values of $b, c$ such that the curve $C$ is singular. This can be done easily because the equation $f_{2}=0$ defines genus 0 curve (in weighted homogenous space) with the parametrization $(b,c)=\Big(\frac{1}{12}(3 a^2+d^2),\frac{1}{144}(9 a^3+9 a d^2-2 d^3)\Big)$. In this case the genus of the curve $C$ is zero and the curve is irreducible. It is possible to parameterize the curve $C$ by rational functions. The resulting curve obtained by this method is of genus 0. However, it has no real rational points and thus cannot be used in our situation.

\bigskip

In order to prove the second part of our theorem we put $x_{4}=a$ and then necessarily $x_{3}=-x_{1}-x_{2}$. We then have $s_{1}(\overline{X}_{4})=a$ and
\begin{equation*}
s_{2}(\overline{X}_{4})=a^2+2(x_{1}^2+x_{1}x_{2}+x_{2}^2),\quad s_{4}(\overline{X}_{4})=a^4+2(x_{1}^2+x_{1}x_{2}+x_{2}^2)^2.
\end{equation*}
Computation of the resultant
\begin{equation*}
\op{Res}_{x_{2}}(s_{2}(\overline{X}_{4})-b,s_{4}(\overline{X}_{4})-c)=4(3a^4 - 2 a^2 b + b^2 - 2 c)^2
\end{equation*}
leads us to $c=\frac{1}{2}(3 a^4 - 2 a^2 b + b^2)$. After this substitution we get an equation of the conic $Q':\;a^2-b+2(x_{1}^2+x_{1}x_{2}+x_{2}^2)=0$. The substitution $b = \frac{1}{2}(2 a^2 + d^2)$, where $d$ is rational parameter, guarantees that $Q'$ has rational point $(x_{1},x_{2})=\Big(0,\frac{d}{2}\Big)$. In this case $Q'$ is parametrized by rational functions
\begin{equation*}
x_{1}=-\frac{2 d t}{t^2+3},\quad x_{2}=\frac{d (t-1) (t+3)}{2(t^2+3)}.
\end{equation*}
These expressions together with $x_{3}=-x_{1}-x_{2},\;x_{4}=a$ solve the system
\begin{equation}\label{specsol2}
s_{1}(\overline{X}_{4})=a,\quad s_{2}(\overline{X}_{4})=\frac{1}{2}(2 a^2 + d^2),\quad s_{3}(\overline{X}_{4})=\frac{1}{8}(8 a^2 + d^4).
\end{equation}
This proves the second statement of our theorem.

\bigskip
Finally, we consider the case $(e_{1},e_{2},e_{3})=(-1,1,2)$. This time we take $x_{4}=\frac{1}{a}$ and thus $x_{3}=-x_{1}-x_{2}-\frac{1}{a}+b$. We then have
\begin{equation*}
s_{-1}(x_{1},x_{2},x_{3},x_{4})-a=\frac{(1-a b)(x_1+x_2)+a(x_1^2+x_2 x_1+x_2^2)}{x_1 x_2(1-a b+a(x_1+x_2))}.
\end{equation*}
We take now $x_{2}=tx_{1}$, where $t$ is a rational parameter, and this allows to solve the above equation with respect to $x_{1}$. We get
\begin{equation*}
x_{1}=\frac{(t+1)(ab-1)}{a(t^2+t+1)},\quad\quad x_{2}=tx_{1}
\end{equation*}
together with expression for $x_{3}$ given by
\begin{equation*}
x_{3}=\frac{(1-ab)t}{a(t^2+t+1)}.
\end{equation*}
With $x_{i}$ for $i=1,2,3,4$ given above we get
\begin{equation*}
s_{2}(x_{1},x_{2},x_{3},x_{4})-c=\frac{a^2b^2-a^2c-2ab+2}{a^2},
\end{equation*}
and thus it is enough to take $c=\frac{2 - 2 a b + a^2 b^2}{a^2}$.

We consider the case $n\geq 5$. We are interested in finding rational numbers $A, B, C$ such that the system
\begin{equation}\label{nsys}
 s_{e_{1}}(\overline{X}_{n})=A,\quad s_{e_{2}}(\overline{X}_{n})=B,\quad s_{e_{3}}(\overline{X}_{n})=C
\end{equation}
has infinitely many rational solutions. Let $t_{2},\ldots,t_{n-3}$ be any sequence of rational numbers and put $p_{k}=s_{k}(t_{2},\ldots,t_{n-3})$ for $k=-1,1,2,3,4$. We first consider the case when $(e_{1},e_{2},e_{3})\in \{(1,2,3), (1,2,4)\}$.  We use the solutions of the system (\ref{mainsys}) with $n=4$. More precisely, if $A$ is a given rational number then we can find infinitely many rational solutions of the system (\ref{mainsys}) with $n=4, a=A-p_{1}$ and corresponding $b, c$ given by the right hand sides of (\ref{specsol1}) and (\ref{specsol2}) respectively. We then take $A=a+p_{e_1}, B=b+p_{e_2}, C=c+p_{e_3}$ and immediately get that for each $j\in\N$ the $n$-tuples
\begin{equation*}
x_{1}=x'_{1,j},\;x_{2}=x'_{2,j},\;x_{3}=x'_{3,j},\;x_{4}=x'_{4,j},\;x_{i}=t_{i}\quad\mbox{for}\quad i=5,\ldots,n
\end{equation*}
is a solutions of the system (\ref{nsys}). Here, the quadruples $(x'_{1,j},\ldots,x'_{4,j}), j\in\N$, are solutions of the system (\ref{mainsys}) with $n=4$.

In the case of $(e_{1},e_{2},e_{3})=(-1,1,2)$ and any pair of rational numbers $A, B$ we can find solutions of the system (\ref{mainsys}) with $n=4, a=A-p_{-1}, b=B-p_{1}$ and corresponding $c$ given by $c=\frac{2 - 2 a b + a^2 b^2}{a^2}$. Similar argument as in the previous cases shows that for $A=a+p_{-1}, B=b+p_{1}$ and $C=c+p_{2}$,  the system (\ref{nsys}) has infinitely many rational solutions.

\end{proof}

\begin{rem}
{\rm An inquisitive reader can ask why we choose $x_{4}=a$ in the case of $(e_{1},e_{2},e_{3})=(1,2,4)$ or $x_{4}=\frac{1}{a}$ in the case of $(e_{1},e_{2},e_{3})=(-1,1,2)$. The answer is simple. In the course of our investigations we used the same method as in the case of $(e_{1},e_{2},e_{3})=(1,2,3)$ and we deduced that $x_{4}$ need to be constant. In order to shorten the proof of the Theorem \ref{123thm} we decided not to repeat the whole reasoning.}
\end{rem}

\begin{rem}
 {\rm The result concerning the system (\ref{123sys}) with $(e_{1},e_{2},e_{3})=(1,2,3)$ with $n=4$ is particular interesting in the light of the result of Dickson \cite[Par. 35, p. 55]{Dic}. Indeed,  he gave the method which allows to find all integer solutions of the Diophantine system
\begin{equation*}
 s_{1}(\overline{X}_{4})=s_{1}(\overline{Y}_{4}),\quad s_{2}(\overline{X}_{4})=s_{2}(\overline{Y}_{4}),\quad s_{3}(\overline{X}_{4})=s_{3}(\overline{Y}_{4}),\quad
\end{equation*}
in eight variables. Here $\overline{Y}_{4}=(y_{1},y_{2},y_{3},y_{4})$. }
\end{rem}

From the Theorem \ref{123thm} immediately deduce the following.

\begin{cor}
For each positive integer $k$ there is at least $k$ $n$-tuples of integers with the same sum of $i$-th powers for $i=1,2,3$. Moreover, we can choose these $n$-triples in such way that the common sum is divisible by $a$, where $a$ is a given positive integer.  Similar result holds for $i=1,2,4$ and $i=-1,1,2$.
\end{cor}

Our results suggest the following general questions concerning the existence of rational solutions of systems involving sums of powers.

\begin{ques}
Let $m\in\N$ and let $(e_{1},\ldots,e_{m})$ be a given sequence of nonzero integers. Is it possible to find a positive integer $n$ and a sequence $(a_{1},\ldots, a_{m})$ of rational numbers such that the system of Diophantine equations
\begin{equation}\label{generalsys}
s_{e_{i}}(\overline{X}_{n})=a_{i}\quad\mbox{for}\quad i=1,2,\ldots,m,
\end{equation}
has infinitely many solutions in rational numbers?
\end{ques}

Let us fix $n\geq 3$ and $m=n-1$. It would be very interesting to characterize all sequences of the exponents $e_{1},\ldots,e_{n-1}$ and corresponding sequences of rational numbers $a_{1},\ldots, a_{n-1}$ such that the curve, say $C$, defined by the system (\ref{generalsys}) has genus $\leq 1$. This approach was used in our proofs. However, it seems that this is a rather difficult question. For example for $n=3$ and $m=2$ for generic choice of $a, b$ the curve defined by (\ref{generalsys}) is of degree $d=e_{1}e_{2}$ and has genus $g$ which satisfy the inequality $2g-2\leq e_{1}e_{2}(e_{1}+e_{2}-4)$ (this is consequence of an adjunction formula
\cite[II, Prop. 8.20 and Exer. 8.4]{Ha}). However, one can check using exactly the same method as in the proof of Theorem \ref{123thm} that sometimes genus of the curve $C$ is low. For example, if we take $e_{1}=2, e_{2}=4$, then for generic choice of $a, b$ the genus of the curve $C$ is 3 (the reason why it is not 9 is simple - the polynomial defining the curve $C$ is a square of a quartic polynomial and thus the genus of $C$ is equal to 3). However, if we take $b=\frac{a^2}{2}$ with $a=2(d^2+d+1)$ then the curve $C$ is rational with parametrization
\begin{equation*}
x_{1}=\frac{t^2+2dt-d-1}{t^2-t+1},\quad x_{2}=\frac{d t^2-2(d+1)t+1}{t^2-t+1},\quad x_{3}=\frac{(d+1)t^2-2 t-d}{t^2-t+1}
\end{equation*}
which leads us to the identities
\begin{equation*}
x_{1}^2+x_{2}^2+x_{3}^2=2(d^2+d+1),\quad x_{1}^4+x_{2}^4+x_{3}^4=2(d^2+d+1)^2.
\end{equation*}

We thus proved the following result which is probably known but we have been unable to find it in the related literature.

\begin{thm}
For each $k$ there are at least $k$ triples of positive integers with the same sum of squares and the same sum of fourth powers.
\end{thm}

We also note the following simple result.

\begin{thm}
Let $a, b\in\Q\setminus\{0\}$. Then the system of Diophantine equations
\begin{equation}\label{14sys}
x_{1}+x_{2}+x_{3}=a,\quad x_{1}^4+x_{2}^4+x_{3}^4=b
\end{equation}
has only finitely many rational solutions.
\end{thm}
\begin{proof}
We first homogenize the equations defining the system (\ref{14sys}) by taking $(x_{1},x_{2},x_{3})=(x/t,y/t,z/t)$. Next, we eliminate $t$ variable from the first equation and get $t=(x+y+z)/a$. Thus, we are working with projective quartic curve in $\mathbb{P}^2(\R)$ given by the equation
\begin{equation}\label{curve14}
C_{a,b}:\;a^4(x^4+y^4+z^4)=b(x+y+z)^4.
\end{equation}
It is clear that if $C_{a,b}$ has infinitely many rational points then necessarily the genus of the curve is $\leq 1$. We prove that for each $a, b\in\Q\setminus\{0\}$ the genus is $\geq 2$ and apply Faltings theorem.

Let $F=F(x,y,z)$ be the polynomial defining the curve $C_{a,b}$. We first check for which $a, b\in\Q\setminus\{0\}$ the curve $C_{a,b}$ is singular. In order to do this we compute the intersection of the ideal $I=<F, \partial_{x}F, \partial_{y}F, \partial_{z} F>$ with $\Z[a,b,z]$ and get
\begin{equation*}
I\cap\Z[a,b,z]=<-a^8(a^4-27b)(a^8+27b^2)z^7>.
\end{equation*}
From our assumption we see that $a(a^8+27b^2)=0$ is impossible. Moreover, if $z=0$ then $x=y=0$ or $x/y$ is constant and there is only finitely many rational points on $C_{a,b}$. We thus take $b=a^4/27$. In particular we note that if $a, b\in\Q\setminus\{0\}$ and $b\neq a^4/27$ then the genus of $C_{a,b}$ is 3. From the Faltings theorem we immediately deduce that the set of rational points on $C_{a,b}$ is finite.

We left with the case $b=a^4/27$. Without loss of generality we assume that $a=1$ and consider the curve $C_{1, 1/27}$. The curve $C_{1,1/27}$ has one ordinary singular point $S=[1:1:1]$ and it is birationally equivalent with the hyperelliptic curve
\begin{equation*}
C':\;Y^2=-6\Big(x^3-\frac{3}{2}x^2+\frac{9}{8} x-\frac{5}{16}\Big)^2-\frac{441}{32} \Big(x-\frac{1}{2}\Big)^2-\frac{63}{32}
\end{equation*}
(of genus 2) with the mapping $\psi:\;C_{1,1/27}\ni (x,y,z)\mapsto \Big(\frac{p}{r},\frac{q}{r^3}\Big)\in C$, where
\begin{align*}
&p=-\frac{1}{3} (x-2y+z),\\
&q=\frac{1}{9}(-8z^3+3(x+y)z+3(x + y)^2-(x + y) (8 x^2 - 11 x y + 8 y^2)),\\
&r=\frac{1}{3} (x+y-2 z).
\end{align*}
The inverse mapping $\psi^{-1}: C\rightarrow C_{1,1/27}$ takes the form $\psi^{-1}(X,Y)=(x,y,z)$, where
\begin{equation*}
\frac{x}{z}=\frac{Y+3(X-1)(2X^2-X+2)}{Y+9X^2-9X+6},\quad \frac{y}{z}=\frac{Y-3X(2X^2-3X+3)}{Y+9X^2-9X+6}.
\end{equation*}
We note that the map $\psi$ is defined outside the set $S=\{[x:y:z]\in C_{1,1/27}(\Q):\;x+y-2z=0\}$. However, if $x+y-2z=0$ then $F(x,y,z)=\frac{27}{16}(x-y)^2 (7 x^2+10 x y+7 y^2)$ and thus $x=y=z$. This implies that $S=\{[1:1:1]\}$. If $x+y-2z\neq 0$ then $\psi$ is well defined and because the set of rational points on $C$ is empty we see that the set of rational points on the curve $C_{1,1/27}$ contains only the point $[1:1:1]$. Our theorem is proved.
\end{proof}

\begin{rem}
{\rm Let us note that for $a=0$ it is easy to find a rational number $b$ such that the system (\ref{14sys}) has infinitely many rational solutions. Indeed, in this case $x_{3}=-x_{1}-x_{2}$ and we get the well known identity $x_{1}^4+x_{2}^4+(x_{1}+x_{2})^4=2(x_{1}^2+x_{1}x_{2}+x_{2}^2)^2$. Taking $b=p^2+pq+q^2$ for some $p, q\in\Q$ it is an easy task to solve the equation $x_{1}^2+x_{1}x_{2}+x_{2}^2=b$ (this equation defines a genus 0 curve with known rational point) and get infinitely many solutions of the system (\ref{14sys}).  }
\end{rem}

\begin{rem}
{\rm Let us also note that we did not prove that it is impossible to find for each $k\in\N$ some non-zero rational numbers $a, b$  such that system (\ref{14sys}), has at least $k$ nontrivial solutions. The reason is simple. We do not know any general results which guarantees that the number of rational point on the corresponding curve $C_{a,b}$ is bounded independent of $a, b$.}
\end{rem}

\bigskip

\noindent {\bf Acknowledgments}
The author express his gratitude to the referee for a careful reading of the manuscript
and valuable suggestions, which improved the quality of the paper.

\bigskip

\noindent Jagiellonian University, Faculty of Mathematics and Computer Science, Institute of Mathematics, {\L}ojasiewicza 6, 30 - 348 Krak\'{o}w, Poland;
 email: {\tt maciej.ulas@uj.edu.pl}

\end{document}